\documentclass[10pt]{article}
\usepackage{amsmath,amsfonts,latexsym, amssymb, fullpage}
\usepackage{mathrsfs}
\usepackage{nicefrac}
\usepackage{graphicx}
\usepackage{amssymb, amsthm, tikz, xcolor}

\newtheorem{theorem}{Theorem}
\newtheorem{definition}[theorem]{Definition}

\newtheorem{lemma}[theorem]{Lemma} 

\definecolor{light-gray}{gray}{0.6}
\definecolor{dark-gray}{gray}{0.3}

 \author{Jeffrey J. Beyerl
\thanks{Department of Mathematics, University of Central Arkansas, Conway, AR 7203
 \mbox{ email: \textsf{jbeyerl@uca.edu}}}
 \and 
 Wayne Wallace
\thanks{Department of Mathematics, University of Central Arkansas, Conway, AR 7203
 \mbox{ email: \textsf{twallace6@cub.uca.edu}}}
 }

\title{Improper Interval Graphs and the Corresponding Minimal Forbidden Interval Subgraphs}
\date{ }

\begin{document}

\maketitle
\begin{abstract}
An interval graph is considered \emph{improper} if and only if it has a representation such that an interval contains another interval. Previously \cite{Beyerl} these have been investigated in terms of balance and minimal forbidden interval subgraphs for the class of 1-improper interval graphs. This paper investigates the minimal forbidden interval subgraphs further, generalizing results to all $p$-improper interval graphs. It is apparent that there are many different types of possible minimal forbidden subgraphs that fall into four broad categories. 
\end{abstract}

\section{Introduction}

An $interval$ $graph$ is a finite, simple graph, $G = (V, E)$ if and only if there is a representation $\alpha : v \longrightarrow I_v$ of vertices in $G$ to intervals, $I_v$, on the real line such that $vw \in E \Leftrightarrow I_v \cap I_w \neq \emptyset$. Interval graphs were first discussed by Hajos \cite{Hajos} in 1957 and expanded on by Lekkerkerker and Boland \cite{Lekkerkerker} in 1962 and Gilmore and Hoffman \cite{Gilmore} in 1964. There have been many expositions and reports on interval graphs since then. Most prominently, Roberts \cite{Roberts} investigated proper interval graphs and provided the complete classification that they are precisely the interval graphs that do not contain the subgraph, $K_{1,3}$. This characterization was later generalized by Proskurowski and Telle \cite{Proskurowski} to $q$-proper interval graphs. A $q$-proper interval graph is an interval graph where no interval is properly contained by more than $q$ others. Beyerl and Jamison \cite{Beyerl} would later explore containment restrictions in the opposite direction, with $p$-improper interval graphs. A $p$-improper interval graph is an interval graph where no interval contains more than $p$ other intervals.
 
A minimal forbidden interval subgraph for $p$-improper interval graphs is an interval graph that is at least $p+1$-improper but not $p$-improper such that all proper subgraphs are $p$-improper. The impropriety imp($G$) of $G$ is defined as the smallest $p$ such that $G$ has a $p$-improper representation. For any particular vertex $z$, the impropriety of the vertex, denoted imp$_\alpha(z)$, is the number of vertices properly contained in $z$. The maximum imp$_\alpha(z)$ is called the impropriety of $G$, imp ($G$). The local components of $G$ are the components of $G\backslash \lbrace z \rbrace$, where $z$ is the vertex of maximum impropriety in $G$. The support of a local component of $G$ is the union of all intervals in the local component. The weight wt($z$) of some vertex $z$ in $G$ is the sum of the $n$-2 smallest orders of the non-exterior local components. If the impropriety of $G$ is equal to the weight of $G$, the graph is said to be balanced. 

Previously the notion of a balanced improper interval graph was introduced \cite{Beyerl}. While somewhat technical, the idea of a balanced improper interval graph is that everything that contributes to the impropriety of the graph has a necessarily simple structure whenever an interval representation is constructed. In our current presentation we will utilize the same idea but relax it a little: instead of a necessary quantifier we use an existential quantifier. In particular we will define the notion of confined and unconfined side components where unconfined side components will be those in which there is a representation that has a simple structure. Using these notions we then give a characterization of the minimal forbidden interval subgraphs for the class of $p$-improper interval graphs.

\section{Side Components}

First we address the idea of a \textit{side component} and a \textit{potential side component}. Formally defined below, a side component can be thought of as the local components that are on the left or right of a representation. Potential side components are the local components that are side components in some representation. In \cite{Beyerl} the notion of a \textit{basepoint} was defined for balanced graphs. We now extend this notion to include all $p$-improper interval graphs by labeling a vertex whose support necessarily contains the support of $p$ other vertices. This definition for a basepoint is independent of representation.

\begin{definition}
Let $\alpha$ be a representation of some interval graph $G$. Let $H_0, ...,H_m$ be the local components of a basepoint of $G$ with respective support $I_0, I_1, ..., I_m$ ordered from left to right. Call $H_0$ and $H_m$ the side components.
\end{definition}

There are four main types of side components we will be concerned with in this paper determined by two binary properties:  exterior and non-exterior, as well as unconfined and confined side components. 

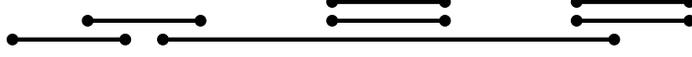
\begin{figure}
\begin{center}
\begin{tikzpicture}
\draw [ultra thick, cap=round] (-3,0) -- (3,0);
\draw [ultra thick] (-3,0) circle [radius=0.05];;
\draw [ultra thick] (3,0) circle [radius=0.05];;

\draw [ultra thick, cap=round] (-4, 0.25) -- (-2.5, 0.25);
\draw [ultra thick] (-4,0.25) circle [radius=0.05];;
\draw [ultra thick] (-2.5,0.25) circle [radius=0.05];;

\draw [ultra thick, cap=round] (2.5, 0.25) -- (4, 0.25);
\draw [ultra thick] (2.5,0.25) circle [radius=0.05];;
\draw [ultra thick] (4,0.25) circle [radius=0.05];;

\draw [ultra thick, cap=round] (-5,0) -- (-3.5,0);
\draw [ultra thick] (-5,0) circle [radius=0.05];;
\draw [ultra thick] (-3.5,0) circle [radius=0.05];;

\draw [ultra thick, cap=round] (2.5,.5) -- (4,.5);
\draw [ultra thick] (2.5,.5) circle [radius=0.05];;
\draw [ultra thick] (4,.5) circle [radius=0.05];;

\draw [ultra thick, cap=round] (-0.75,.25) -- (0.75,.25);
\draw [ultra thick] (-0.75,.25) circle [radius=0.05];;
\draw [ultra thick] (0.75,.25) circle [radius=0.05];;

\draw [ultra thick, cap=round] (-0.75,0.5) -- (0.75,0.5);
\draw [ultra thick] (-0.75,0.5) circle [radius=0.05];;
\draw [ultra thick] (0.75,0.5) circle [radius=0.05];;
\end{tikzpicture}
\caption{An interval representation of a 2-improper interval graph. The left most side component is an unconfined, exterior side component. The right most side component is an unconfined, non-exterior side component.}
\end{center}
\end{figure}

\begin{definition}
Let $G$ be a $p$-improper interval graph and $H$ be a side component of $G$. Let $z$ be the basepoint of $G$. $H$ is called an exterior side component if it contains at least one vertex of distance at least 2 from $z$. Similarly, a side component that contains no vertices of distance 2 from the basepoint is considered non-exterior. 
\end{definition}

Note that from the above we see that given a representation there are only two ways to construct an exterior local component: either the component's support is on the left side of the basepoint, or on the right side. It cannot be contained within the support of the basepoint.

Confined and unconfined side components are defined as follows.

\begin{definition}
Let $G$ be a $p$-improper interval graph and $H$ be a potential side component of $G$. If there is a representation in which $H$ contains no vertex that contributes to the impropriety of $G$, then $H$ is called an unconfined side component. On the other hand, a potential side component that necessarily does contain one or more vertices that contribute to the impropriety of $G$ is considered confined.
\end{definition}

Now we can classify side components into four sets; side components that are unconfined and exterior, side components that are unconfined and non-exterior, side components that are confined and exterior, and side components that are confined and non-exterior. We can investigate each of these types of side components in turn. 

Let us begin by noting a necessary technical lemma. 

\begin{lemma}\label{NoExtra}
Let $G$ be a minimal forbidden subgraph for the class of $p$-improper interval graphs. Then let $F$ be the graph formed by the addition of any additional vertices to $G$. Then $F$ is not minimal. 
\end{lemma}
\begin{proof}
$F$ contains $G$ as a subgraph. $G$ is minimal for the class of $p$-improper interval graphs. By the definition of minimal, $F$ cannot be minimal, as desired. 
\end{proof}

\section{Main Theorems}

Now let first examine unconfined, exterior side components. We want to show the following claim:

\begin{theorem}\label{UnconfinedExterior}
Let $G$ be a minimal forbidden interval subgraph for the class of $p$-improper interval graphs. Let $H$ be an unconfined, exterior side component of $G$. Then $H$ consists of precisely an exterior vertex and a vertex connecting it to the basepoint of $G$.
\end{theorem}

\begin{proof}
$H$ must have an exterior vertex because of the definition of exterior side component. Now recall that an exterior vertex is a vertex that is of distance at least two from the basepoint. Since the exterior vertex is of distance two from the basepoint, there must be at least one vertex between our exterior vertex and the basepoint; however, one connecting vertex is sufficient. By Lemma \ref{NoExtra}, there can be no additional vertices in $H$ that are not contained in the basepoint. Because $H$ is unconfined there can be no additional verticies in $H$ that are contained in the basepoint. We conclude that there is exactly one vertex connecting the exterior vertex and the basepoint. 
\end{proof}

From here, we can move to unconfined, non-exterior side components, but first, we must introduce some additional terminology. 

\begin{definition}
Let $G$ be an interval graph, with a representation, $\alpha$, and let $B$ be a set of vertices of $G$ who are adjacent to the basepoint of $G$, but not properly contained by the basepoint. Also assume that the vertices in $B$ are not adjacent to any exterior vertices of $G$. Call $B$ a blocking set of $G$. Note that a blocking set of $G$ need not be an entire local component of $G$.
\end{definition}

Note that a blocking set depends on the representation. As with potential side components, we define a potential blocking set as a set of vertices of $G$ that can be made into a blocking set of $G$ by changing the representation of $G$. Given a representation, we also define interior potential blocking sets as those potential blocking sets that are contained in potential side components that are not actual side components. This will be used to identify those blocking sets that can definitely be switched with an actual blocking set.

\begin{theorem}\label{UnconfinedNonExterior}
Let $G$ be a minimal forbidden interval subgraph for the class of $p$-improper interval graphs and $\alpha$ be a minimal representation of $G$. Let $H$ be an unconfined, non-exterior side component of $G$. Then $H$ consists of a clique that is of maximal order among all interior potential blocking sets of $G$.  
\end{theorem}

\begin{proof}
Since $H$ is unconfined, it must have no vertices that contribute to the impropriety of $G$, and since $H$ is non-exterior, all vertices must be adjacent to the basepoint of $G$. Thus, $H$ is a blocking set of $G$ as well as a clique. 

Now assume there was an interior potential blocking set of $G$ with order greater than $|H|$. This larger blocking set can be switched with $H$ which would lower the impropriety of $G$, giving rise to the contradiction that $\alpha$ is not minimal. Thus, $H$ must be of maximal order among the interior potential blocking sets. 


Hence, all unconfined, non-exterior side components in minimal forbidden interval subgraphs for $p$-improper interval graphs must be of the form described above. 
\end{proof}

Figure 1 gives two examples of unconfined side components, both exterior and non-exterior. 

Now we can move on to the cases of confined side components, which are more complicated than their unconfined counterparts. Examples of these can be seen in Figure 2. We'll begin with the confined, non-exterior side component. 

\begin{theorem}
Let $G$ be a minimal forbidden interval subgraph for the class of $p$-improper interval graphs and $\alpha$ be a minimal representation of $G$. Let $H$ be a confined, non-exterior side component of $G$. Then $H$ consists of a blocking set, $B$ of maximum order among the interior potential blocking sets, and some set of vertices that contribute to the impropriety of $G$. Furthermore, if the blocking set has $n$ vertices, some number of these vertices, ranging from 1 to $n-1$ must be adjacent to at least one of the vertices in the contributory set of vertices. 
\end{theorem}

\begin{proof}
Let $H$ be as described and without loss of generality, assume $H$ is the rightmost side component of $G$. Since $H$ is confined, it must contain some set of vertices that contribute to the impropriety of $G$. There must also be a blocking set of vertices. To see this, let $a$ be the rightmost endpoint of the rightmost vertex in $H$ that contributes to the impropriety of $G$, and let $b$ be the rightmost endpoint of the basepoint. If there is no blocking set in $H$, then there exists another representation of $G$ with no changes other than $a>b$. This new representation has a lower impropriety, a contradiction. 

Using the same reasoning as used in Theorem \ref{UnconfinedNonExterior}, this blocking set must be of maximum order among interior potential blocking sets of $G$. 

Now we will show that 1 to $n-1$ vertices in the blocking set must be adjacent to at least one vertex in the contributing set. Assume that no vertices in the blocking set are adjacent to a vertex in the contributory set. Then the blocking set and the contributory set are separate local components of $G$, and so we would not be in the confined case. Now assume that all $n$ vertices in the blocking set are adjacent to at least one vertex in the contributory set. Again, let $a$ be the rightmost endpoint of this rightmost vertex in the contributory set and $b$ be the rightmost endpoint of the basepoint. There exists another representation of $G$ where $a>b$, which means that $\alpha$ is not minimal, which is a contradiction. Thus, the number of vertices in the blocking set adjacent to a vertex in the contributory set must range from 1 to $n-1$. 

By lemma \ref{NoExtra}, no further vertices can be added to $H$, so all unbalanced, non-exterior side components in minimal forbidden interval subgraphs for $p$-improper interval graphs must be of the form described above. 
\end{proof}

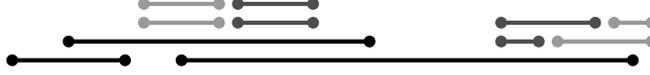
\begin{figure}
\begin{center}
\begin{tikzpicture}
\draw [ultra thick, cap=round] (-3,0) -- (3,0);
\draw [ultra thick] (-3,0) circle [radius=0.05];;
\draw [ultra thick] (3,0) circle [radius=0.05];;
\draw [ultra thick, cap=round] (-5.25,0) -- (-3.75,0);
\draw [ultra thick] (-5.25,0) circle [radius=0.05];;
\draw [ultra thick] (-3.75,0) circle [radius=0.05];;
\draw [ultra thick, cap=round] (-4.5,0.25) -- (-0.5,0.25);
\draw [ultra thick] (-4.5,0.25) circle [radius=0.05];;
\draw [ultra thick] (-0.5,0.25) circle [radius=0.05];;
\draw [light-gray, ultra thick, cap=round] (-3.5,0.5) -- (-2.5,0.5);
\draw [light-gray, ultra thick] (-3.5,0.5) circle [radius=0.05];;
\draw [light-gray, ultra thick] (-2.5,0.5) circle [radius=0.05];;

\draw [light-gray, ultra thick, cap=round] (-3.5,0.75) -- (-2.5,0.75);
\draw [light-gray, ultra thick] (-3.5,0.75) circle [radius=0.05];;
\draw [light-gray, ultra thick] (-2.5,0.75) circle [radius=0.05];;
\draw [dark-gray, ultra thick, cap=round] (-2.25,0.75) -- (-1.25,0.75);
\draw [dark-gray, ultra thick] (-2.25,0.75) circle [radius=0.05];;
\draw [dark-gray, ultra thick] (-1.25,0.75) circle [radius=0.05];;

\draw [dark-gray, ultra thick, cap=round] (-2.25,0.5) -- (-1.25,0.5);
\draw [dark-gray, ultra thick] (-2.25,0.5) circle [radius=0.05];;
\draw [dark-gray, ultra thick] (-1.25,0.5) circle [radius=0.05];;
\draw [light-gray, ultra thick, cap=round] (2.0,0.25) -- (3.25,0.25);
\draw [light-gray, ultra thick] (2.0,0.25) circle [radius=0.05];;
\draw [light-gray, ultra thick] (3.25,0.25) circle [radius=0.05];;

\draw [light-gray, ultra thick, cap=round] (2.75,0.5) -- (3.25,0.5);
\draw [light-gray, ultra thick] (2.75,0.5) circle [radius=0.05];;
\draw [light-gray, ultra thick] (3.25,0.5) circle [radius=0.05];;
\draw [dark-gray, ultra thick, cap=round] (1.25,0.25) -- (1.75,0.25);
\draw [dark-gray, ultra thick] (1.25,0.25) circle [radius=0.05];;
\draw [dark-gray, ultra thick] (1.75,0.25) circle [radius=0.05];;

\draw [dark-gray, ultra thick, cap=round] (1.25,0.5) -- (2.5,0.5);
\draw [dark-gray, ultra thick] (1.25,0.5) circle [radius=0.05];;
\draw [dark-gray, ultra thick] (2.5,0.5) circle [radius=0.05];;
\end{tikzpicture}

\caption{An interval representation of a 4-improper interval graph, $G$. The left side component is an exterior, confined side component, and the right is confined, non-exterior. The dark gray intervals mark contributory sets and the light gray intervals are blocking sets.}
\end{center}
\end{figure}

The last case to consider is the case of confined exterior side components. 

\begin{theorem}
Let $G$ be a minimal forbidden interval subgraph for the class of $p$-improper interval graphs and $\alpha$ be a minimal representation of $G$. Let $H$ be a confined, exterior side component of $G$. Then $H$ consists of the following: 
\begin{itemize}
	\item Exactly one exterior vertex
	\item Exactly one vertex connecting the exterior vertex to the basepoint
	\item Some set of vertices that contribute to the impropriety of $G$, at least one of which is adjacent to the connecting vertex. 
	\item Some clique of vertices adjacent to the basepoint of $G$ but not properly contained by the basepoint, excluding the connecting vertex. This set of vertices is the blocking set of $H$ and as long as there are two side components, this blocking set is of maximum order among the potential blocking sets contained in $H$.
\end{itemize}	
Furthermore, if there are $n$ vertices in the blocking set of $H$, then no more than $n-1$ of these vertices can be adjacent to at least one vertex in the contributory set of vertices of $H$. All vertices in the blocking set of $H$ must be adjacent to the basepoint of $G$, and no vertex of the blocking set of $H$ can be adjacent to the exterior vertex. 
\end{theorem}

\begin{proof}
Let $H$ be as described and without loss of generality assume $H$ is the rightmost side component. Since $H$ is exterior, there must be one exterior vertex, and one connecting vertex. Since $H$ is confined, it must contain some set of vertices that contribute to the impropriety of $G$. Along with the contributing set, there must also be a blocking set within $H$.  

To see this, let $a$ be the rightmost endpoint of the rightmost vertex in the contributory set of $H$. Let $b$ be the rightmost endpoint of the basepoint of $G$. If there were no blocking set of $H$, then there exists another representation of $G$ in which nothing is changed except $a>b$: this would reduce the impropriety of $G$ and so is not possible because $G$ is minimal forbidden. Thus there must exist a blocking set within $H$. 

Furthermore by the same reasoning we can say that there must be at least one vertex in the blocking set of $H$ that is not adjacent to the contributory set of vertices. Call this vertex $q$. 

Furthermore, every vertex in the blocking set of $H$ must be adjacent to the basepoint of $G$. To prove this, assume that there exists a vertex in the blocking set of $H$ that is a distance of two from the basepoint of $G$. Call this vertex $w$. Now look at the subgraph of $G$ created by removing the following vertices: the original exterior vertex, the connecting vertex, and every vertex in the blocking set of $H$ other than $q$ and $w$. This subgraph is a forbidden interval subgraph, with the remnants of the blocking set of $H$ forming an unconfied exterior side component. As $G$ is minimal forbidden this is impossible. Hence every vertex in the blocking set is adjacent to the basepoint.

It remains to be shown that this blocking set of $H$ must also be a clique of maximum order among all potential blocking sets of $H$.

First we show that the blocking set of $H$ must be a clique. To see this, assume that there is a vertex in the blocking set of $H$ that is not adjacent to some other vertex in $H$. We see that one of these vertices is either contributing to the impropriety of $G$, or that one of these vertices is not adjacent to the basepoint of $G$. By definition, a vertex cannot be in both the contributory and blocking sets. In the other case, the blocking set of $H$ contains a vertex of distance two from the basepoint of $G$, which would create two exterior vertices. We've already shown that this cannot occur, and so the blocking set of $H$ must indeed be a clique.

To show the blocking set must have maximum order among all potential blocking sets of $H$, assume there is a potential blocking set of $H$ larger than the actual blocking set of $H$ in $\alpha$. There exists a representation of $G$ such that the larger potential blocking set and the original blocking set are switched. This reduces the impropriety of $G$, meaning that $G$ was not minimal forbidden, a contradiction. Therefore, the original blocking set of $H$ must have maximum order among the potential blocking components of $H$. 

Therefore, all confined, exterior side components in minimal forbidden interval subgraphs for $p$-improper interval graphs must be of the form described above. 
\end{proof}

The final theorem counts the number of minimal forbidden interval subgraphs for the class of $p$-improper interval graphs when both side components are unconfined. 

\begin{theorem}
There are $3\left( 2^{{p\choose 2}}\right)-2$ minimal forbidden interval subgraphs for the class of $p-1$-improper interval graphs that have two unconfined side components. 
\end{theorem}

\begin{proof}
Consider the class of minimal forbidden subgraphs for the class of $p-1$-improper interval graphs that have two unconfined side components. Let $G$ be some arbitrary member of this class. Because the side components are unconfined, these graphs must be exactly $p$-improper, so there are exactly $p$ vertices that contribute to the impropriety. In a minimal representation there can be no other vertices that contribute to the impropriety because $G$ is minimal forbidden. Now consider the subgraph $H$ formed by these contributory vertices. There are no structural requirements on $H$, so there are $2^{p \choose 2}$ possible graphs that $H$ could be. 

$G$ can either have two non-exterior side components, two exterior side components, or one exterior side component and one non-exterior side component. From Theorem \ref{UnconfinedNonExterior} and Theorem \ref{UnconfinedExterior}, we see that each of these three possible configuration of side components are distinct.

All of these configurations are forbidden, and all but two of them are minimal. The two that are not minimal are the two cases where $H$ has no edges and one or both of the side components are exterior. This is because we can remove the exterior vertex in the side component and the remaining vertex is a blocking set. Thus, both of $G$'s side components in this case must be non-exterior, which removes two cases. This gives us our final count of $3\left(2^{{\binom{p}{2}}}\right)-2$ minimal forbidden interval subgraphs for the class of $p$-improper interval graphs that have two unconfined side components. 
\end{proof}

\bibliographystyle{plain}
\bibliography{References}

\end{document}